\documentclass[12pt]{amsart}

\usepackage{geometry}
\usepackage[utf8]{inputenc}
\usepackage[english]{babel}
\usepackage{amssymb}
\usepackage{amsmath}
\usepackage{amsthm}
\usepackage{amscd}
\usepackage{enumitem}
\usepackage{graphicx}
\usepackage{tikz}

\def\ID{\mathbb D}

\def\IR{\mathbb R}
\def\IN{\mathbb N}

\def\IC{\mathbb C}

\def\IT{\mathbb T}

\newcommand{\intd}{\text{d}}
\newcommand{\intervaloo}[2]{\mathopen{(}#1,#2\mathclose{)}}
\newcommand{\intervalco}[2]{\mathopen{[}#1,#2\mathclose{)}}
\newcommand{\intervaloc}[2]{\mathopen{(}#1,#2\mathclose{]}}
\newcommand{\intervalcc}[2]{\mathopen{[}#1,#2\mathclose{]}}

\theoremstyle{plain}
\newtheorem{theorem}{Theorem}[section]
\newtheorem{lemma}[theorem]{Lemma}
\newtheorem{proposition}[theorem]{Proposition}

\theoremstyle{definition}
\newtheorem{definition}[theorem]{Definition}

\theoremstyle{remark}
\newtheorem{remark}[theorem]{Remark}

\numberwithin{equation}{section}

\title[Rate of growth of random analytic functions]{Rate of growth of random analytic functions, with an application to linear dynamics}
\author{Kevin Agneessens and Karl-G. Grosse-Erdmann}
\address{Kevin Agneessens, 
D\'epartement de Math\'ematique, Universit\'e de Mons, 20 Place du Parc, 7000 Mons, Belgium}
\email{agneessens.kevin@gmail.com}
\address{Karl-G. Grosse-Erdmann, 
D\'epartement de Math\'ematique, Universit\'e de Mons, 20 Place du Parc, 7000 Mons, Belgium}
\email{kg.grosse-erdmann@umons.ac.be}

\thanks{The first author was a Research Fellow of the Fonds de la Recherche Scientifique - FNRS. The second author was supported by the Fonds de la Recherche Scientifique - FNRS under Grant n\textsuperscript{o} CDR J.0078.21.}


\begin{document}

\begin{abstract}
We obtain Wiman-Valiron type inequalities for random entire functions and for random analytic functions on the unit disk that improve a classical result of Erd\H{o}s and Rényi and recent results of Kuryliak and Skaskiv. Our results are then applied to linear dynamics: we obtain rates of growth, outside some exceptional set, for analytic functions that are frequently hypercyclic for an arbitrary chaotic weighted backward shift.
\end{abstract}

\maketitle

\section{Introduction}\label{s-intro}

We are interested in what might be called the probabilistic Wiman-Valiron theory. Our investigation leads to an extension of a classical result of Erd\H{o}s and Rényi, and to an improvement of recent results of Kuryliak and Skaskiv, see Theorems \ref{t-main1} and \ref{t-main2} below. We also present an application of our work to linear dynamics. 

Let us start by explaining the background.

\subsection{Wiman-Valiron theory} The classical theory of Wiman and Valiron studies the relationship between the maximum modulus and the maximum term of an entire function. More precisely, let $f(z) = \sum_{n=0}^\infty a_n z^n$, $z\in \IC$, be a non-constant entire function. Denoting, as usual, by
\begin{equation}\label{eq-mm}
M_f(r) = \max_{|z|=r}|f(z)| \ \text{ and } \ \mu_f(r) = \max_{n\geq 0}|a_n|r^n
\end{equation}
the maximum modulus and the maximum term of $f$ for $r\geq 0$, respectively, then one of the main results of the Wiman-Valiron theory states that, for any $\delta>0$, there is a (measurable) set $E\subset [0,\infty)$ of finite logarithmic measure and some $C>0$ such that
\begin{equation}\label{eq-wm}
M_f(r)\leq C \mu_f(r)\big(\log \mu_f(r)\big)^{\frac{1}{2}+\delta}, \ r\notin E,
\end{equation}
see Wiman \cite{Wiman1914} and Valiron \cite{Valiron1918}, \cite{Valiron1920}, \cite[p.\ 106]{Valiron1923}; recall that $E$ is of finite logarithmic measure if $\int_{E\cap\intervalco{1}{\infty}} \tfrac{1}{r}\intd r<\infty$. Inequality \eqref{eq-wm} was later improved by Rosenbloom \cite{Rosenbloom1962} who showed that, for any $\delta>0$, there is a set $E\subset [0,\infty)$ of finite logarithmic measure and some $C>0$ such that
\begin{equation}\label{eq-wm2}
M_f(r)\leq C \mu_f(r)\big(\log \mu_f(r)\big)^{\frac{1}{2}}\big(\log\log\mu_f(r)\big)^{1+\delta}, \ r\notin E.
\end{equation}
Further strengthenings can be found in \cite{Rosenbloom1962}, \cite{Hayman1974} and \cite[Theorem 6.23]{Hayman1989}; see also \cite{GrosseErdmann2024} for a survey.

For introductions to the Wiman-Valiron theory we refer to \cite{Hayman1974}, \cite[Section 6.5]{Hayman1989} and \cite{JankVolkmann1985}.

\subsection{Probabilistic Wiman-Valiron theory} A probabilistic variant of inequality \eqref{eq-wm} was first considered by Lévy \cite{Levy1930}, see also \cite[p.\ 55]{ErdosRenyi1969}. Let $(X_n)_{n\geq 0}$ be an independent sequence of Steinhaus random variables defined on a probability space $(\Omega, \mathcal{A},\mathbb{P})$; recall that a complex random variable is Steinhaus if it is uniformly distributed on the unit circle $\IT$. For an entire function $f(z)= \sum_{n=0}^\infty a_n z^n$, let us consider 
\[
\sum_{n=0}^\infty a_n X_n(\omega) z^n, \ z\in \IC, \omega\in \Omega.
\]
This defines a random entire function. Then Lévy showed the following, under some regularity assumptions on the coefficients $a_n$: for any $\delta>0$ there exists almost surely a set $E\subset [0,\infty)$ of finite logarithmic measure and some $C>0$ such that
\begin{equation}\label{eq-pwm}
\max_{|z|=r} \Big|\sum_{n=0}^\infty a_n X_n z^n\Big| \leq C \mu_f(r)\big(\log \mu_f(r)\big)^{\frac{1}{4}+\delta}, \ r\notin E.
\end{equation}
In other words, randomizing the coefficients allows to lower the exponent $\frac{1}{2}$ in \eqref{eq-wm} to $\frac{1}{4}$ in \eqref{eq-pwm}; in the recent literature, this is referred to as \textit{Lévy's phenomenon}, see for example \cite{KuryliakSkaskivZrum2014} and \cite{Kuryliak2017}.

Erd\H{o}s and Rényi \cite{ErdosRenyi1969} showed that \eqref{eq-pwm} holds for any non-constant entire function $f$ if the $X_n$ are independent Rademacher random variables, that is, if they are uniformly $\pm 1$-distributed. More importantly, they also obtained the same improvement for inequality \eqref{eq-wm2}: for any $\delta>0$ there exists almost surely a set $E\subset [0,\infty)$ of finite logarithmic measure and some $C>0$ such that
\begin{equation}\label{eq-pwm2}
\max_{|z|=r} \Big|\sum_{n=0}^\infty a_n X_n z^n\Big| \leq C \mu_f(r)\big(\log \mu_f(r)\big)^{\frac{1}{4}}\big(\log\log\mu_f(r)\big)^{1+\delta}, \ r\notin E.
\end{equation}

For our intended applications to the theory of linear dynamics these results are, however, not good enough: we need to consider complex random variables $X_n$ of full support, for example complex Gaussian random variables; see Section \ref{s-lindyn}. It turned out that a similar problem had already been posed by O. B. Skaskiv who had asked whether Lévy’s phenomenon also holds for unbounded random variables, see \cite[p.\ 12]{Kuryliak2017}. Kuryliak \cite[Theorem 3, Corollary 1]{Kuryliak2017} has obtained a positive answer for an independent centred subgaussian sequence of random variables (see Definition \ref{definitionSubgaussian}), however with exponent $\frac{3}{2}+\delta$ instead of $1+\delta$ in \eqref{eq-pwm2}. 

Our first main result confirms Lévy's phenomenon for independent centred subgaussian sequences of random variables.

\begin{theorem}\label{t-main1}
Let $f(z)=\sum_{n= 0}^\infty a_{n}z^n$ be a non-constant entire function and $(X_n)_{n\geq 0}$ an independent centred subgaussian sequence. Then $\sum_{n= 0}^\infty a_{n}X_{n}z^n$ defines almost surely an entire function. Moreover, for every $\delta>0$, there exists a set $E\subset\intervalco{0}{\infty}$ of finite logarithmic measure and a constant $C>0$ such that, almost surely, there exists $r_0>0$ such that
\[
\max_{|z|=r} \Big|\sum_{n=0}^{\infty}a_{n}X_{n}z^n\Big| \leq C\mu_{f}(r)\big(\log\mu_{f}(r)\big)^{\frac{1}{4}}\big(\log\log \mu_{f}(r)\big)^{1+\delta}
\]
for every $r\geq r_0$, $r\notin E$.
\end{theorem}

Note that the exceptional set and the constant are independent of $\omega\in\Omega$.

This theorem contains the result of Erd\H{o}s and Rényi as a special case, and it improves that of Kuryliak.

\subsection{Wiman-Valiron theory in the disk}
The Wiman-Valiron theory in the unit disk $\ID$ was initiated by K\"ov\'ari \cite{Kovari1966}. Let $f(z)=\sum_{n=0}^\infty a_nz^n$, $z\in \ID$, be a non-constant analytic function in $\ID$. The maximum modulus and maximum term functions are defined as in \eqref{eq-mm}. Unlike for the plane, there does not seem to be a canonical version of the Wiman-Valiron inequality \eqref{eq-wm}: the place of an additional term $\frac{1}{1-r}$ and the way the exceptional set is measured vary in the literature, see \cite[Theorem 1]{Kovari1966}, \cite{Suleimanov1980}, \cite[Theorem 2]{FentonStrumia2009}, \cite[Theorem 1]{SkaskivKuryliak2020}; see also \cite{GrosseErdmann2024} for a survey. 

Our starting point was a special case of a Wiman-Valiron inequality established by Sule\u{\i}\-ma\-nov \cite{Suleimanov1980}, see also \cite[p.\ 83]{KuryliakSkaskivShapovalovska2016} and \cite{GrosseErdmann2024}: for any $\delta>0$, there is a set $E\subset [0,1)$ of finite logarithmic measure and some $C>0$ such that
\[
M_f(r)\leq C \frac{\mu_f(r)}{(1-r)^{1+\delta}}\Big(\log \frac{\mu_f(r)}{1-r}\Big)^{\frac{1}{2}+\delta}, \ r\notin E;
\]
here, $E$ is said to be of finite logarithmic measure if $\int_{E} \tfrac{1}{1-r}\intd r<\infty$.

This was recently improved by Skaskiv and Kuryliak \cite{SkaskivKuryliak2020} who essentially showed that one even has that 
\begin{equation}\label{eq-wmD2}
M_f(r)\leq C \frac{\mu_f(r)}{1-r}\Big(\log \frac{1}{1-r}\Big)^{\frac{1}{2}+\delta}\Big(\log \frac{\mu_f(r)}{1-r}\Big)^{\frac{1}{2}}\Big(\log \log\frac{\mu_f(r)}{1-r}\Big)^{1+\delta}, \ r\notin E;
\end{equation}
see Theorem \ref{t-SkaskivKuryliak} below for more details.

\subsection{Probabilistic Wiman-Valiron theory in the disk} 
Kuryliak, Skaskiv and Skaskiv have shown that variants of the Lévy phenomenon also hold in the unit disk, see \cite[Theorem 2.3]{KuryliakSkaskivSkaskiv2020}, \cite[Theorem 2, Corollary 2]{KuryliakSkaskiv2022}, and \cite[Theorem 1, Corollary 1]{KuryliakSkaskiv2023}. 
The first two results concern centred bounded random variables. In \cite{KuryliakSkaskiv2023}, centred subgaussian random variables are considered, however under an additional assumption (see Remark \ref{r-KS23}). In our second main result, using a different proof, we obtain an improved inequality, and we show that the additional assumption in \cite{KuryliakSkaskiv2023} can be dropped. 

\begin{theorem}\label{t-main2}
Let $f(z)=\sum_{n= 0}^\infty a_{n}z^n$ be a non-constant analytic function in $\ID$ and $(X_n)_{n\geq 0}$ an independent centred subgaussian sequence. Then $\sum_{n= 0}^\infty a_{n}X_{n}z^n$ defines almost surely an analytic function in $\ID$. Moreover, for every $\delta>0$, there exists a set $E\subset\intervalco{0}{1}$ of finite logarithmic measure and a constant $C>0$ such that, almost surely, there exists $r_0 \in (0,1)$ such that
\[
\max_{|z|=r} \Big|\sum_{n=0}^{\infty}a_{n}X_{n}z^n\Big| \leq C \frac{\mu_f(r)}{(1-r)^{\frac{1}{2}}}\Big(\log\frac{1}{1-r}\Big)^{\frac{3}{4}+\delta}\Big(\log\frac{\mu_f(r)}{1-r}\Big)^{\frac{1}{4}}\Big(\log\log\frac{\mu_f(r)}{1-r}\Big)^{1+\delta}
\]
for $r_0\leq r<1$, $r\notin E$.
\end{theorem}

Kuryliak and Skaskiv seem to conjecture in \cite[p.\ 755]{KuryliakSkaskiv2022} that the optimal exponent of $\log\frac{1}{1-r}$ should be $\frac{1}{4}+\delta$, at least for certain random variables.

\subsection{Notation and organisation} 
We will use $a\lesssim b$ to indicate an inequality up to a multiplicative constant that is independent of the running variable in the expressions $a$ and $b$. The notation $a\asymp b$ means $a\lesssim b$ and $b\lesssim a$.

In inequalities like \eqref{eq-wm} one can omit the constant $C$, which is often done in the literature. For the sake of coherence we prefer to state all results with a constant $C$.

The paper is organised as follows. We start with some results that do not distinguish between the plane and the disk. We then deduce Lévy phenomena in the plane, including Theorem \ref{t-main1} (Section \ref{s-LevyC}), and in the disk, including Theorem \ref{t-main2} (Section \ref{s-LevyD}). Section \ref{s-unif} unifies the previous results. Section \ref{s-lindyn} presents applications to linear dynamics.

\begin{remark}
In the probabilistic Wiman-Valiron results of Sections \ref{s-intro} to \ref{s-unif} we demand for brevity that the sequence $(X_n)_n$ of complex random variables is independent. Of course, it would in each case suffice to demand that both the real parts $(\text{Re}\, X_n)_n$ and the imaginary parts $(\text{Im}\, X_n)_n$ are independent, as is done in the work of Kuryliak and Skaskiv. To see this, it suffices to apply the results to theses two sequences separately.
\end{remark}

\section{Some preliminary results}\label{s-prelim}
We introduce briefly the objects of our study. Throughout, all (real or complex) random variables are defined on a probability space $(\Omega,\mathcal{A}, \mathbb{P})$. 

\begin{definition}\label{definitionSubgaussian}
A random variable $X$ is \textit{subgaussian} if there exist constants $K,\tau > 0$ such that $\mathbb{P}(|X| > t) \leq Ke^{-t^2/\tau^2}$ for all $t \geq 0$. 

A sequence $(X_n)_{n\geq 0}$ of random variables is \emph{subgaussian} if each $X_n$, $n\geq 0$, is subgaussian with the same constants $K$ and $\tau$.
\end{definition}

Note that a real random variable $X$ is centred subgaussian if and only if there is some $\sigma>0$ such that, for all $\lambda\in \IR$,
\begin{equation}\label{eq-kah}
\mathbb{E}(e^{\lambda X}) \leq e^{\sigma^2\lambda^2},
\end{equation}
see \cite[pp.\ 4-6]{Kahane1960} and \cite[Exercice 10, p.\ 82]{Kahane1985}; see also \cite{Agneessens2023b}. Thus a centred complex subgaussian random variable is one for which the real and imaginary parts satisfy \eqref{eq-kah}.

Subgaussian random variables have been considered by Kahane \cite{Kahane1960}, \cite{Kahane1985}, who takes \eqref{eq-kah} as definition and therefore demands that they are centred. Of course, any Gaussian random variable and any bounded random variable is subgaussian.

We call the sequence $(X_n)_{n\geq 0}$ centred if each $X_n$ is centred. The following result on centred subgaussian sequences is crucial for our work; it is proved in Kahane \cite[Chapter 6, Theorem 2]{Kahane1985}. For any complex trigonometric polynomial $q(t) = \sum_{n=-N}^N a_n e^{int}$ we write $\|q\|_\infty = \max_{t\in \intervalcc{0}{2\pi}}|q(t)|$. 

\begin{lemma}[Kahane]\label{rogHCELemmasupTrigoPoly}
Let $(X_n)_{n\geq 0}$ be an independent centred subgaussian sequence. Then there exists a constant $C>0$ such that, for any positive integers $M,N\geq 1$, and any sequence $(q_n)_{n=0}^{M}$ of complex trigonometric polynomials of degree less than or equal to $N$,
\[
\mathbb{P}\Bigg(
\bigg\|\sum_{n=0}^{M} X_{n}q_{n}\bigg\|_\infty\geq  C\sqrt{\log N}\Big(\sum_{n=0}^{M}\|q_{n}\|_{\infty}^2\Big)^\frac{1}{2}
\Bigg)
\leq \frac{C}{N^2}.
\]
\end{lemma}

Note that the constant $C$ may only depend on the constants $K,\tau>0$ of the subgaussian sequence.

We now consider functions $f(z)=\sum_{n=0}^\infty a_n z^n$ that are analytic for $|z|<R$, $0<R\leq \infty$. Let $(X_n)_{n\geq }$ be a sequence of complex random variables. If they are uniformly bounded, then, for every $\omega\in \Omega$, $\sum_{n=0}^\infty a_n X_n(\omega) z^n$ also has radius of convergence at least $R$. This is also almost surely true if $(X_n)_{n\geq 0}$ is subgaussian, as we will now see.

\begin{proposition}\label{rogHCELemmaWellDefined}
Let $0<R\leq \infty$, and let $f(z)=\sum_{n=0}^\infty a_nz^n$ be analytic for $|z|<R$. Let $(X_n)_{n\geq 0}$ be a subgaussian sequence. Then the random series $\sum_{n=0}^\infty a_nX_nz^n$ has almost surely radius of convergence at least $R$.
\end{proposition}

\begin{proof}
By assumption we have that $\limsup_{n\to\infty}|a_n|^{1/n}\leq 1/R$. Thus, for every $0<r<R$, there exist $0<\rho<1$ and $n_0\geq 1$ such that, for every $n\geq n_0$, $\sqrt{\log n}\ |a_n|r^n\leq \rho^n$. This implies that $\sum_{n=1}^\infty\sqrt{\log n}\ a_nz^n$ has radius of convergence at least $R$. 

It suffices to prove the claim for real subgaussian sequences $(X_n)_{n\geq 0}$. Fix $c>0$, and let $K>0$ and $\tau>0$ be constants associated to this sequence. Then we have that
\[
\sum_{n=1}^\infty \mathbb{P}\big(|X_n|> c\sqrt{\log n}\big) \leq K\sum_{n=1}^\infty e^{-c^2(\log n)/\tau^2} =\sum_{n=1}^\infty\frac{K}{n^{c^2/\tau^2}}.
\]
If $c^2>\tau^2$ then $\sum_{n=1}^\infty\mathbb{P}(|X_n| > c\sqrt{\log n})$ converges. It follows from the Borel-Cantelli lemma that, almost surely, $|X_n|\leq c\sqrt{\log n}$ for $n$ large enough. Therefore also $\sum_{n=0}^\infty a_nX_nz^n$ has almost surely radius of convergence  at least $R$.
\end{proof}

In particular, for any subgaussian sequence $(X_n)_{n\geq 0}$ and any function $f(z)=\sum_{n=0}^\infty a_nz^n$ that is analytic in $\IC$ or in $\ID$, the (formal) series $\sum_{n=0}^\infty a_n X_n z^n$ is almost surely well-defined and analytic in $\IC$ or $\ID$, respectively. This proves the first assertions in Theorems \ref{t-main1} and \ref{t-main2}.

We now prepare the proofs of the main parts of these theorems by some lemmas that do not depend on the radius of convergence.
The first two lemmas are inspired by \cite[Lemma 6.15]{Hayman1989}, \cite[Lemma 1]{KuryliakSkaskivShapovalovska2016} and \cite[Lemma 3.2 and p.\ 143]{KuryliakSkaskivSkaskiv2020}.

\begin{lemma}\label{rogHCELemmaBoundG}
Let $0< R\leq\infty$. Let $g:\intervalco{0}{R}\to\intervalco{0}{\infty}$ be a continuously differentiable increasing function with $\lim_{r\to R}g(r)>1$ and $h:\intervalco{\rho}{R}\to\intervalco{0}{\infty}$ be a continuous increasing function with $\int_{\rho}^R\frac{h(r)}{r}\emph{\intd} r=\infty$ for some $\rho\in \intervalco{0}{R}$. Then, for every $\delta>0$, there exists an open set $E\subset\intervalco{0}{\infty}$ with $\int_{E\cap\intervalco{\rho}{R}}\frac{h(r)}{r}\emph{\intd} r<\infty$ such that, for every $r> \rho$, $r\notin E$, 
\[
\frac{\emph{d}}{\emph{d} r}\log g(r)\leq\frac{h(r)}{r}\big(\log g(r)\big)^{1+\delta}.
\]
\end{lemma}

\begin{proof}
By assumption on $g$ we can assume $\rho$ so large that $g(r)\geq \eta$ for some $\eta>1$ and all $r\geq \rho$. Let $E\subset\intervaloo{\rho}{R}$ be the set where the inequality of the lemma does not hold. Since both sides of the inequality are continuous, the set $E$ is open. Using the change of variables $x=\log g(r)$ (note that $\log g$ need not be strictly increasing, see \cite[p.\ 156]{Rudin87}) we obtain that
\[
\int_{E}\frac{h(r)}{r}\intd r\leq\int_{E}\frac{\frac{\text{d}}{\text{d} r}\log g(r)}{\big(\log g(r)\big)^{1+\delta}}\intd r \leq\int_{\log\eta}^{\infty}\frac{1}{x^{1+\delta}}\intd x<\infty,
\]
which gives the desired restriction for $E$.
\end{proof}

\begin{lemma}\label{rogHCELemmaseriesEstimate}
Let $0<R\leq \infty$. Let the real power series $g(r)=\sum_{n=0}^\infty a_nr^n$, with $a_n\geq 0$ for all $n\geq 0$, have radius of convergence at least $R$, and suppose that $\lim_{r\to R}g(r)>1$. Let $h:\intervalco{\rho}{R}\to\intervalco{0}{\infty}$ be a continuous increasing function with $\int_{\rho}^R\frac{h(r)}{r}\emph{\intd} r=\infty$ for some $\rho\in \intervalco{0}{R}$. Then, for every $\delta>0$, there exists an open set $E\subset\intervalco{0}{R}$ with $\int_{E\cap\intervalco{\rho}{R}}\frac{h(r)}{r}\emph{\intd} r<\infty$ such that, for every $r> \rho$, $r\notin E$,
\[
\sum_{n=0}^{\infty}na_nr^n \leq h(r)g(r)\big(\log g(r)\big)^{1+\delta}.
\]
\end{lemma}

\begin{proof}
First notice that, for every $r\in \intervaloo{0}{R}$, one has $\frac{\text{d}}{\text{d} r} g(r)=r^{-1}\sum_{n= 0}^\infty na_nr^{n}$
and thus
\[
\sum_{n=0}^{\infty}na_nr^{n}=r\frac{\text{d}}{\text{d} r}g(r)=rg(r)\frac{\text{d}}{\text{d} r}\log g(r)
\]
for $r$ sufficiently large. Then the result follows from Lemma \ref{rogHCELemmaBoundG}.
\end{proof}

\section{Lévy's phenomenon in the plane}\label{s-LevyC}

We first study the rate of growth for random power series on $\IC$ with subgaussian coefficients. Our aim here is to give a proof of Theorem \ref{t-main1}.  

The main ideas in this section come from Erd\H{o}s and R\'{e}nyi \cite{ErdosRenyi1969}, Steele \cite{Steele1987}, Kuryliak \cite{Kuryliak2017}, and Kuryliak, Skaskiv and Skaskiv \cite{KuryliakSkaskivSkaskiv2020}.

Let $f(z)=\sum_{n=0}^\infty a_nz^n$ be an entire function. We have already defined its maximum term $\mu_f(r)$, $r\geq 0$. In addition, we will need the expressions
\[
S_f(r)=\Big(\sum_{n= 0}^\infty|a_n|^2r^{2n}\Big)^{\frac{1}{2}} \text{ and } G_f(r)=\sum_{n=0}^{\infty}|a_n|r^{n}.
\]
Note that $\mu_f(r)\leq S_f(r)\leq G_f(r)$. It might also be of interest that, by Parseval's identity,
\[
S_f(r) = \Big(\frac{1}{2\pi}\int_0^{2\pi}|f(re^{it})|^2 \intd t\Big)^{\frac{1}{2}},
\]
which is also denoted as $M_2(f,r)$. In the present context, the function $S_f(r)$ has already appeared in Erd\H{o}s and R\'{e}nyi \cite{ErdosRenyi1969} and Steele \cite{Steele1987}. 

All three functions are continuous in $r$ (see \cite[Satz 4.2]{JankVolkmann1985} for $\mu_f$), and if $f$ is non-constant then they tend to infinity as $r\to\infty$.

Recall also that a (measurable) set $E\subset\intervalco{0}{\infty}$ is of \textit{finite logarithmic measure} if $\int_{E\cap\intervalco{1}{\infty}} \frac{1}{r}\intd r<\infty$. Obviously, in order to show that some property holds outside a set of finite logarithmic measure, it suffices to prove that there exists a set of finite logarithmic measure such that the property holds outside this set and for $r$ sufficiently large. 

Applying the Rosenbloom inequality \eqref{eq-wm2} to the entire function $z\to\sum_{n=0}^\infty |a_n|z^n$ we obtain the following; note that its maximal modulus function is $G_f$ and its maximum term is $\mu_f$.

\begin{lemma}\label{l-roseng}
Let $f(z)=\sum_{n=0}^\infty a_nz^n$ be a non-constant entire function. Then, for every $\delta>0$, there is an open set $E\subset\intervalco{0}{\infty}$ of finite logarithmic measure and a constant $C>0$ such that, for any $r\notin E$,
\[
G_f(r)\leq C \mu_f(r)\big(\log \mu_f(r)\big)^{\frac{1}{2}}\big(\log\log\mu_f(r)\big)^{1+\delta}.
\]
\end{lemma}

Here, $E$ can be chosen to be open because both sides of the inequality are continuous.

In the sequel, for ease of writing, we use the notation
\[
\|f\|_r =M_f(r) = \max_{|z|=r}|f(z)|.
\]

\begin{lemma}\label{rogHCELemmaProbaEstimate}
Let $f(z)=\sum_{n=0}^\infty a_nz^n$ be a non-constant entire function and $(X_n)_{n\geq 0}$ an independent centred subgaussian sequence. Let $\alpha>1$ and $\delta>0$. Then $\sum_{n=0}^\infty a_{n}X_{n}z^n$ defines almost surely an entire function, and there exists a constant $C>0$ and an open set $E\subset\intervalco{0}{\infty}$ of finite logarithmic measure such that, for any $r\notin E$,
\[
\mathbb{P}\Bigg(
\bigg\|\sum_{n=0}^{\infty}a_{n}X_{n}z^n\bigg\|_{r} \geq C\sqrt{\log N}S_{f}(r) \Bigg)
\leq\frac{C}{N^{2\alpha}}
\]
whenever $N\geq\big(\log\mu_{f}(r)\big)^{\frac{3}{2}+\delta}$.
\end{lemma}

\begin{proof}
The first assertion is given by Proposition \ref{rogHCELemmaWellDefined}. 

We next apply Lemma \ref{rogHCELemmaseriesEstimate} to the power series $G_f(r)=\sum_{n=0}^\infty |a_n|r^n$, to the function $h(r)=1$, $r\geq 0$, and to $\delta/2$; note that the exceptional set is then of finite logarithmic measure. Let $E\subset\intervalco{0}{\infty}$ be the open set of finite logarithmic measure that is the union of the open sets in that lemma and in Lemma \ref{l-roseng}, also applied for $\delta/2$. Then we have, for any $r\notin E$ sufficiently large,
\begin{equation}\label{eq-ang}
\sum_{n=0}^{\infty}n|a_n|r^n \leq G_f(r)\big(\log G_f(r)\big)^{1+\frac{\delta}{2}}
\end{equation}
and 
\begin{equation}\label{eq-wm2g}
G_f(r)\lesssim \mu_f(r)\big(\log \mu_f(r)\big)^{\frac{1}{2}}\big(\log\log\mu_f(r)\big)^{1+\frac{\delta}{2}}.
\end{equation} 

Let $\alpha>1$. We define $B_{n}:=\{|X_n|> n^{1-1/\alpha}\}\subset\Omega$ for $n\geq 1$. Since $(X_n)_n$ is subgaussian there is some $\tau>0$ such that, for $n\geq 1$,
\[
\mathbb{P}(B_n)\lesssim e^{- n^{2-2/\alpha}/\tau^2}\lesssim\frac{1}{n^3}.
\]
For any real $N\geq 1$, define $B(N):=\bigcup_{n>N^{\alpha}}B_n$. Then
\[
\mathbb{P}(B(N))\leq\sum_{n>N^{\alpha}}\mathbb{P}(B_n)\lesssim\sum_{n>N^{\alpha}}\frac{1}{n^3}\lesssim\frac{1}{N^{2\alpha}}.
\]
On the complement of $B(N)\subset \Omega$ we get for $r\geq 0$ that
\[
\Big\|\sum_{n>N^{\alpha}}a_nX_nz^n\Big\|_{r} \leq\sum_{n>N^{\alpha}}|X_n||a_n|r^n
\leq\sum_{n>N^{\alpha}}n^{1-1/\alpha}|a_n|r^n\notag\leq N^{-1}\sum_{n>N^{\alpha}}n|a_n|r^n.
\]

Now let $r_0>0$ satisfy $\mu_f(r_0)>e$. Let $r\geq r_0$, $r\notin E$, and let $N\geq\big(\log \mu_f(r)\big)^{\frac{3}{2}+\delta}>1$ be a real number. Then we have on the complement of $B(N)$, with \eqref{eq-ang} and \eqref{eq-wm2g},
\begin{align*}
\Big\|\sum_{n>N^{\alpha}}a_nX_nz^n\Big\|_{r} &\leq N^{-1}G_f(r)\big(\log G_f(r)\big)^{1+\frac{\delta}{2}}\\
&\lesssim N^{-1}\mu_f(r)\big(\log \mu_f(r)\big)^{\frac{1}{2}}\big(\log\log\mu_f(r)\big)^{1+\frac{\delta}{2}}\big(\log \mu_f(r)\big)^{1+\frac{\delta}{2}}\\
&\lesssim N^{-1}\mu_f(r)\big(\log \mu_f(r)\big)^{\frac{3}{2}+\delta} \leq \mu_f(r) \leq S_f(r).
\end{align*}
Therefore there is a constant $C_1>0$ such that, if $r\geq r_0$, $r\notin E$ and $N\geq\big(\log \mu_f(r)\big)^{\frac{3}{2}+\delta}$, then
\[
\mathbb{P}\Big(\Big\|\sum_{n>N^{\alpha}}a_nX_nz^n\Big\|_{r} >C_1S_f(r) \Big) \leq\mathbb{P}(B(N)) \lesssim\frac{1}{N^{2\alpha}}.
\]

By Lemma \ref{rogHCELemmasupTrigoPoly} applied to $q_n(t)=a_nr^ne^{int}$, $t\in[0,2\pi]$, $n\geq 0$, with $M$ and $N$ given by $\lfloor N^{\alpha} \rfloor$, we have on the other hand that there is a constant $C_2>0$ such that
\[
\mathbb{P}\Big(\Big\|\sum_{0\leq n\leq N^{\alpha}}a_nX_nz^n\Big\|_{r}\geq C_2\sqrt{\log N}S_f(r)\Big)
\lesssim\frac{1}{\lfloor N^{\alpha} \rfloor^{2}}.
\]
Altogether there is some constant $C>0$ such that
\[
\mathbb{P}\Big(\Big\|\sum_{n=0}^{\infty}a_nX_nz^n\Big\|_{r} \geq C\sqrt{\log N}S_f(r) \Big)
\lesssim\frac{1}{N^{2\alpha}}+\frac{1}{N^{2\alpha}}
\]
if $r\geq r_0$, $r\notin E$ and $N\geq\big(\log \mu_f(r)\big)^{\frac{3}{2}+\delta}$. This completes the proof.
\end{proof}

We next need a lemma, versions of which seem to appear in every proof of Lévy's phenomenon; see, for example, Erd\H{o}s-Rényi \cite[p.\ 49]{ErdosRenyi1969}, Steele \cite[p.\ 555]{Steele1987} or Kuryliak \cite[Lemma 8]{Kuryliak2017}.

\begin{lemma}\label{rogHCELemmaSubseq}
Let $\varphi: \intervalco{\rho}{\infty}\to\intervalco{1}{\infty}$ be a continuous increasing function such that\linebreak  $\lim_{r\to\infty}\varphi(r)=\infty$, where $\rho\geq 0$. Let $E\subset\intervaloo{\rho}{\infty}$ be an open set of unbounded complement.
Then there exists an infinite set $J\subset\IN$ and an increasing sequence $(r_{k})_{k\in J}$ in $[\rho,\infty)$ such that, for every $k\in J$,
\begin{enumerate}[label=\emph{(\roman*)}]
\item $r_{k}\notin E$,\label{rogHCELemmaSubseq1}
\item $\varphi(r_{k})\geq k$,\label{rogHCELemmaSubseq2}
\item for any $r\geq \rho$ with $r\notin E$ there exists $k\in J$ such that $r\leq r_{k}$ and $\varphi(r_{k})\leq \varphi(r)+1$. \label{rogHCELemmaSubseq3}
\end{enumerate}
\end{lemma}

\begin{proof}
Define for each $k\geq 1$ the possibly empty set
\[
U_{k}:=\big\{r\geq \rho: k\leq \varphi(r)\leq k+1\big\}.
\]
These sets are closed since $\varphi$ is continuous, and bounded since $\lim_{r\to\infty}\varphi(r)=\infty$, and thus they are compact.
Define $J:=\{k\in\IN : U_{k}\setminus E\neq\varnothing\}$. For each $k\in J$, there exists $r_{k}\in U_{k}\setminus E$ such that $r_{k}=\sup (U_{k}\setminus E)$. This gives (i) and (ii). Since $\lim_{r\to\infty}\varphi(r)=\infty$ and $E$ is of unbounded complement, the set $J$ is infinite.

Let $r\geq \rho$. Since $\varphi(\rho)\geq 1$, there exists $k\in\IN$ such that $k\leq \varphi(r)\leq k+1$. If $r\notin E$ then $k\in J$, and $r\leq r_{k}$ by definition of $r_{k}$. By definition of $U_{k}$, we also have $\varphi(r_{k})\leq \varphi(r)+1$. This gives (iii).
\end{proof}

We can now prove the main result of this section, which is stronger than Theorem \ref{t-main1}. First, thanks to the Borel-Cantelli lemma, we will prove the desired inequality for a suitable sequence $(r_k)_{k\geq 1}$ chosen with Lemma \ref{rogHCELemmaSubseq}. The properties of this sequence and the Maximum Principle will then conclude the proof.

\begin{theorem}\label{rogHCEThmMain}
Let $f(z)=\sum_{n= 0}^\infty a_{n}z^n$ be a non-constant entire function and $(X_n)_{n\geq 0}$ an independent centred subgaussian sequence. Then $\sum_{n= 0}^\infty a_{n}X_{n}z^n$ defines almost surely an entire function. Moreover, there exists an open set $E\subset\intervalco{0}{\infty}$ of finite logarithmic measure and a constant $C>0$ such that, almost surely, there exists $r_0>0$ such that
\[
\max_{|z|=r} \Big|\sum_{n=0}^{\infty}a_{n}X_{n}z^n\Big| \leq C\sqrt{\log\log \mu_{f}(r)}S_f(r)
\]
for every $r\geq r_0$, $r\notin E$.
\end{theorem}

\begin{proof}
The first assertion is given by Proposition \ref{rogHCELemmaWellDefined}.

Now let $E\subset\intervalco{0}{\infty}$ be the open set of finite logarithmic measure that is the union of the open set in
 Lemma \ref{rogHCELemmaProbaEstimate}, taken for some $\alpha>1$ and $\delta>0$, and the open set in Lemma \ref{l-roseng} for the same $\delta$. Note that $E$ has an unbounded complement.

By Lemma \ref{rogHCELemmaSubseq} applied to $\varphi=\log S_{f}$ and $\rho\geq 0$ so large that $\log\mu_{f}(\rho)> 1$ and hence $\log S_{f}(\rho) > 1$, we get an infinite set $J\subset\IN$ and an increasing sequence $(r_k)_{k\in J}$ in $[\rho,\infty)$ converging to $\infty$ and satisfying assertions \emph{\ref{rogHCELemmaSubseq1}}, \emph{\ref{rogHCELemmaSubseq2}} and \emph{\ref{rogHCELemmaSubseq3}} of the lemma.

Define for each $k\in J$ the real number
\[
N_k:=\big(\log \mu_f(r_k)\big)^{\frac{3}{2}+\delta}\geq 1
\]
and the set
\[
A_k:=\Big\{\Big\|\sum_{n=0}^{\infty}a_nX_nz^n\Big\|_{r_k}\geq C\sqrt{\log N_k}S_f(r_k)\Big\},
\]
where $C>0$ is the constant of Lemma \ref{rogHCELemmaProbaEstimate}. Then \emph{\ref{rogHCELemmaSubseq1}} of Lemma \ref{rogHCELemmaSubseq}, Lemma \ref{rogHCELemmaProbaEstimate} and the definition of $N_k$ imply that
\[
\sum_{k\in J}\mathbb{P}(A_k) \lesssim\sum_{k\in J}\frac{1}{N_k^{2\alpha}} =\sum_{k\in J}\frac{1}{\big(\log \mu_f(r_k)\big)^{\alpha(3+2\delta)}}.
\]
By Lemma \ref{l-roseng} we have, for every $r\notin E$, that
\[
\mu_f(r)\leq S_f(r)\leq G_f(r)\lesssim \mu_f(r)\big(\log \mu_f(r)\big)^{\frac{1}{2}}\big(\log\log\mu_f(r)\big)^{1+\delta}.
\]
This implies that
\begin{equation}\label{eq-equiv}
\log S_f(r)\asymp\log\mu_f(r) \ \text{for $r\notin E$.}
\end{equation}
Therefore, using \emph{\ref{rogHCELemmaSubseq1}} and \emph{\ref{rogHCELemmaSubseq2}} of Lemma \ref{rogHCELemmaSubseq}, we have that
\[
\sum_{k\in J}\mathbb{P}(A_k)\lesssim\sum_{k\in J}\frac{1}{\big(\log S_f(r_k)\big)^{\alpha(3+2\delta)}}\lesssim\sum_{k=1}^{\infty}\frac{1}{k^{\alpha(3+2\delta)}}<\infty.
\]
This in turn implies by the Borel-Cantelli lemma that, for almost every $\omega\in\Omega$, there exists $k_0(\omega)\in J$ such that, for every $k\in J$ with $k\geq k_0(\omega)$,
\begin{equation}\label{rogHCEthmMainInequalityROGsubsequence}
\Big\|\sum_{n=0}^{\infty}a_nX_n(\omega)z^n\Big\|_{r_k} \leq C\sqrt{\log N_k}S_f(r_k).
\end{equation}

Let $r>r_{k_0(\omega)}$ with $r\notin E$. By \emph{\ref{rogHCELemmaSubseq3}} of Lemma \ref{rogHCELemmaSubseq} there is some $k\in J$ with $k > k_0(\omega)$ such that $r\leq r_k$ and $\log S_f(r_k)\leq \log S_f(r)+1$, hence $S_f(r_k)\leq eS_f(r)$. The Maximum Principle, \eqref{eq-equiv} and \eqref{rogHCEthmMainInequalityROGsubsequence} yield
\begin{align*}
\Big\|\sum_{n=0}^{\infty}a_{n}X_{n}(\omega)z^n\Big\|_{r} &\leq\Big\|\sum_{n=0}^{\infty}a_{n}X_{n}(\omega)z^n\Big\|_{r_{k}}
\leq C\sqrt{\log N_k}S_f(r_k)\\
&\asymp\sqrt{\log\log \mu_{f}(r_k)}S_f(r_{k})\leq\sqrt{\log\log S_{f}(r_k)}S_f(r_{k})\\
&\lesssim\sqrt{\log\log S_{f}(r)}S_f(r)\asymp\sqrt{\log\log \mu_{f}(r)}S_f(r),
\end{align*}
which completes the proof.
\end{proof}

In Kuryliak \cite{Kuryliak2017}, the sequence $(r_k)_{k\geq 1}$ was constructed from the maximum term $\mu_f$. The idea of constructing this sequence from $S_f$ instead comes from \cite{ErdosRenyi1969} and \cite{Steele1987}.

Theorem \ref{rogHCEThmMain} generalizes Theorem 2 of Erd\H{o}s and R\'{e}nyi \cite{ErdosRenyi1969}, who use Rademacher random variables. Indeed, every bounded random variable is subgaussian. In their main result, Theorem 1, Erd\H{o}s and R\'{e}nyi obtain a rate of growth written in terms of the maximum term. We obtain Theorem \ref{t-main1} in the same way.

\begin{proof}[Proof of Theorem \ref{t-main1}]
This result is a direct consequence of Theorem \ref{rogHCEThmMain} and the Wiman-Valiron inequality in the form of Rosenbloom. Indeed, let $\delta>0$, and let $E$ be the union of the sets given by Lemma \ref{l-roseng} and Theorem \ref{rogHCEThmMain}. By Lemma \ref{l-roseng} we have for $r\notin E$ 
\[
S_f^2(r) \leq\mu_f(r)G_f(r) \lesssim \mu_f(r)^2\big(\log \mu_f(r)\big)^{\frac{1}{2}}\big(\log\log\mu_f(r)\big)^{1+\delta}.
\]
It remains to apply Theorem \ref{rogHCEThmMain}.
\end{proof}

\section{Lévy's phenomenon in the disk}\label{s-LevyD}

We now study the rate of growth for random power series on $\ID$ with subgaussian coefficients. Our aim is to prove Theorem \ref{t-main2}. The proof is very similar to that of Theorem \ref{t-main1}, with a slight complication arising from the presence of an additional term $\frac{1}{1-r}$.

Recall that a (measurable) set $E\subset\intervalco{0}{1}$ is of \textit{finite logarithmic measure} if $\int_{E}\frac{1}{1-r}\intd r<\infty$. Again, it will always suffice to show that a property holds outside a set of finite logarithmic measure and for all $r$ close enough to $1$.

The maximum modulus $M_f$ of an analytic function $f$ on $\ID$, its maximum term $\mu_f$ and the functions $S_f$ and $G_f$ are defined exactly in the same way as for entire functions.
 
Applying the Wiman-Valiron inequality of Skaskiv and Kuryliak \cite{SkaskivKuryliak2020}, see \eqref{eq-wmD2}, to $z\to\sum_{n=0}^\infty|a_n|z^n$, we have the following.

\begin{lemma}\label{l-skkurg}
Let $f(z)=\sum_{n=0}^\infty a_nz^n$ be a non-constant analytic function on $\ID$. Then, for every $\delta>0$, there is an open set $E\subset\intervalco{0}{1}$ of finite logarithmic measure and a constant $C>0$ such that, for any $r\in \intervalco{0}{1}$, $r\notin E$,
\[
G_f(r)\leq C \frac{\mu_f(r)}{1-r}\Big(\log \frac{1}{1-r}\Big)^{\frac{1}{2}+\delta}\Big(\log \frac{\mu_f(r)}{1-r}\Big)^{\frac{1}{2}}\Big(\log\log \frac{\mu_f(r)}{1-r}\Big)^{1+\delta}.
\] 
\end{lemma}

Again, $E$ can be chosen to be open because both sides of the inequality are continuous. And we will continue to write $\|f\|_r = M_f(r)=\max_{|z|=r}|f(z)|$.

\begin{lemma}\label{rogHDELemmaprobaTail}
Let $f(z)=\sum_{n=0}^\infty a_nz^n$ be a non-constant analytic function on $\ID$ so that $\lim_{r\to 1}\mu_f(r)>e$, and let $(X_n)_{n\geq 0}$ be an independent centred subgaussian sequence. Let $\alpha>1$ and $\delta>0$. Then $\sum_{n= 0}^\infty a_nX_nz^n$ defines almost surely an analytic function on $\ID$, and there exists a constant $C>0$ and an open set $E\subset\intervalco{0}{1}$ of finite logarithmic measure such that, for any $r\notin E$,
\[
\mathbb{P}\Big(\Big\|\sum_{n=0}^{\infty}a_nX_nz^n\Big\|_{r} \geq C\sqrt{\log N}S_{f}(r)\Big)\leq\frac{C}{N^{2\alpha}}
\]
whenever $N\geq\frac{1}{(1-r)^2}\big(\log\frac{\mu_f(r)}{1-r}\big)^{2+\delta}$.
\end{lemma}

\begin{proof}
The first assertion follows from Proposition \ref{rogHCELemmaWellDefined}.

We next apply Lemma \ref{rogHCELemmaseriesEstimate} to the power series $G_f(r)=\sum_{n=0}^\infty |a_n|r^n$, to the function $h(r)=\frac{1}{1-r}$, and to $\delta/2$; then the exceptional set is of finite logarithmic measure. Let $E\subset\intervalco{0}{1}$ be the open set of finite logarithmic measure that is the union of the open sets in this lemma and in Lemma \ref{l-skkurg}, applied for $\delta/4$. Then we have, for any $r\notin E$ sufficiently large,
\begin{equation}\label{eq-andg}
\sum_{n=0}^{\infty}n|a_n|r^n \leq \frac{1}{1-r}G_f(r)\big(\log G_f(r)\big)^{1+\frac{\delta}{2}}
\end{equation}
and 
\begin{equation}\label{eq-wm2dg}
\begin{split}
G_f(r)&\lesssim \frac{\mu_f(r)}{1-r}\Big(\log \frac{1}{1-r}\Big)^{\frac{1}{2}+\frac{\delta}{4}}\Big(\log \frac{\mu_f(r)}{1-r}\Big)^{\frac{1}{2}}\Big(\log\log \frac{\mu_f(r)}{1-r}\Big)^{1+\frac{\delta}{4}}\\
&\lesssim \frac{\mu_f(r)}{1-r}\Big(\log \frac{\mu_f(r)}{1-r}\Big)^{1+\frac{\delta}{2}}\lesssim \Big( \frac{\mu_f(r)}{1-r}\Big)^{1+\delta},
\end{split}
\end{equation} 
where we have used that $\mu_f(r)\geq 1$ for large $r$.

Let $\alpha>1$. Define $B_n:=\{|X_n|> n^{1-1/\alpha}\}\subset\Omega$ for $n\geq 1$ and $B(N):=\bigcup_{n>N^\alpha}B_n$ for any real $N\geq 1$. Then the argument in the proof of Lemma \ref{rogHCELemmaProbaEstimate} shows that
\[
\mathbb{P}(B(N))\lesssim\frac{1}{N^{2\alpha}}
\]
and that we have on the complement of $B(N)$ for $r\geq 0$ 
\[
\Big\|\sum_{n>N^{\alpha}}a_nX_nz^n\Big\|_{r} \leq\frac{1}{N}\sum_{n>N^{\alpha}}n|a_n|r^n.
\]

Let $r_0>0$ satisfy $\mu_f(r_0)>e$. Let $r\geq r_0$, $r\notin E$, and let $N\geq\frac{1}{(1-r)^2}\big(\log\frac{\mu_f(r)}{1-r}\big)^{2+\delta}>1$. Note that the hypothesis on $\mu_f$ implies that $\lim_{r\to 1}G_f(r)>1$. By \eqref{eq-andg} and \eqref{eq-wm2dg}, we get on the complement of $B(N)$ that
\begin{align*}
\Big\|\sum_{n>N^{\alpha}}a_nX_nz^n\Big\|_{r} &\leq \frac{1}{N}\frac{1}{1-r}G_f(r)\big(\log G_f(r)\big)^{1+\frac{\delta}{2}}\\
&\lesssim \frac{1}{N}\frac{\mu_f(r)}{(1-r)^2}\Big(\log \frac{\mu_f(r)}{1-r}\Big)^{1+\frac{\delta}{2}}\Big(\log\frac{\mu_f(r)}{1-r}\Big)^{1+\frac{\delta}{2}}\\
&= \frac{1}{N}\frac{\mu_f(r)}{(1-r)^2}\Big(\log \frac{\mu_f(r)}{1-r}\Big)^{2+\delta} \leq \mu_f(r) \leq S_f(r).
\end{align*}

Therefore there is a constant $C_1>0$ such that if $r\geq r_0$, $r\notin E$, and $N\geq\frac{1}{(1-r)^2}\big(\log\frac{\mu_f(r)}{1-r}\big)^{2+\delta}$ then
\[
\mathbb{P}\Big(\Big\|\sum_{n>N^{\alpha}}a_nX_nz^n\Big\|_{r}>C_1S_f(r)\Big) \leq\mathbb{P}(B(N))\lesssim\frac{1}{N^{2\alpha}}.
\]

By Lemma \ref{rogHCELemmasupTrigoPoly} applied to $q_n(t)=a_nr^ne^{int}$, $t\in \intervalcc{0}{2\pi}$, $n\geq 0$, with $M$ and $N$ given by $\lfloor N^{\alpha} \rfloor$, we have on the other hand that there is a constant $C_2>0$ such that
\[
\mathbb{P}\Big(\Big\|\sum_{0\leq n\leq N^{\alpha}}a_nX_nz^n\Big\|_{r}\geq C_2\sqrt{\log N}S_f(r)\Big)\lesssim\frac{1}{N^{2\alpha}}.
\]

Altogether there is some constant $C>0$ such that
\[
\mathbb{P}\Big(\Big\|\sum_{n=0}^{\infty}a_nX_nz^n\Big\|_{r} \geq C\sqrt{\log N}S_f(r) \Big)
\lesssim\frac{1}{N^{2\alpha}}+\frac{1}{N^{2\alpha}}
\]
for $r\geq r_0$, $r\notin E$, and $N\geq\frac{1}{(1-r)^2}\big(\log\frac{\mu_f(r)}{1-r}\big)^{2+\delta}$. This completes the proof.
\end{proof}

Due to the additional term $\frac{1}{1-r}$ we will now need a more elaborate version of Lemma \ref{rogHCELemmaSubseq}. In view of Section \ref{s-unif} we formulate it here for arbitrary $R>0$.

\begin{lemma}\label{rogHDELemmasubseq}
Let $0<R\leq \infty$. Let $\varphi,\psi_1,\psi_2: \intervalco{\rho}{R}\to\intervalco{1}{\infty}$ be continuous increasing functions such that $\lim_{r\to R}\varphi(r)=\infty$, where $\rho \in \intervalco{0}{R}$. Let $E\subset\intervaloo{\rho}{R}$ be an open set whose complement has $R$ as limit point. Then there exists an infinite set $J\subset \IN^3$ and a family $(r_{l,k,j})_{(l,k,j)\in J}$ in $[\rho,R)$ such that
\begin{enumerate}[label=\emph{(\roman*)}]
\item for any $(l,k,j)\in J$, $r_{l,k,j}\notin E$,\label{rogHDEsubsequenceLemma1}
\item for any $(l,k,j)\in J$, $l\leq \varphi(r_{l,k,j})\leq l+1$, $k\leq \psi_1(r_{l,k,j})\leq k+1$ and $j\leq \psi_2(r_{l,k,j})\leq j+1$,\label{rogHDEsubsequenceLemma2}
\item for any $r\in \intervalco{\rho}{R}$ with $r\notin E$ there exists $(l,k,j)\in J$ such that $r\leq r_{l,k,j}$, $\varphi(r_{l,k,j})\leq \varphi(r)+1$, $\psi_1(r_{l,k,j})\leq \psi_1(r)+1$ and $\psi_2(r_{l,k,j})\leq \psi_2(r)+1$. \label{rogHDEsubsequenceLemma3}
\end{enumerate}
\end{lemma}

\begin{proof}
Define for each $l,k,j\geq 1$ the possibly empty set
\[
U_{l,k,j}:=\big\{\rho\leq r<R: l\leq \varphi(r)\leq l+1,\ k\leq \psi_1(r)\leq k+1\text{ and } j\leq \psi_2(r)\leq j+1\big\}.
\]
These sets are closed in $[0,R)$ since $\varphi$, $\psi_1$ and $\psi_2$ are continuous, and bounded away from $R$ since $\lim_{r\to R}\varphi(r)=\infty$, and thus they are compact. Define $J:=\{(l,k,j)\in\IN^3: U_{l,k,j}\setminus E\neq\varnothing\}$. For each $(l,k,j)\in J$, there exists $r_{l,k,j}\in U_{l,k,j}\setminus E$ such that $r_{l,k,j}=\sup (U_{l,k,j}\setminus E)$. This shows (i) and (ii). Since $\lim_{r\to R}\varphi(r)=\infty$ and the complement of $E$ has $R$ as limit point, $J$ is an infinite set.

Let $\rho\leq r<R$. Since $\varphi(\rho)\geq 1$, $\psi_1(\rho)\geq 1$ and $\psi_2(\rho)\geq 1$, there exists $(l,k,j)\in \IN^3$ such that $l\leq \varphi(r)\leq l+1$,  $k\leq \psi_1(r)\leq k+1$ and $j\leq \psi_2(r)\leq j+1$. If $r\notin E$ then $(l,k,j)\in J$, and $r\leq r_{l,k,j}$ by definition of $r_{l,k,j}$. By definition of $U_{l,k,j}$, we also have $\varphi(r_{l,k,j})\leq \varphi(r)+1$, $\psi_1(r_{l,k,j})\leq \psi_1(r)+1$ and $\psi_2(r_{l,k,j})\leq \psi_2(r)+1$.
\end{proof}

\begin{remark} For a possible future application let us note that the previous lemma can obviously be extended to any number of functions $\psi_1,\ldots,\psi_n$, $n\geq 2$.
\end{remark}

The next theorem is the main result of this section.

\begin{theorem}\label{rogHDEThmMain}
Let $f(z)=\sum_{n=0}^\infty a_nz^n$ be a non-constant analytic function on $\ID$ and $(X_n)_{n\geq 0}$ an independent centred subgaussian sequence. Then $\sum_{n=0}^\infty a_nX_nz^n$ defines almost surely an analytic function on $\ID$. Moreover, there exists an open set $E\subset\intervalco{0}{1}$ of finite logarithmic measure and a constant $C>0$ such that, almost surely, there exists some $r_0\in \intervaloo{0}{1}$ such that
\[
\max_{|z|=r}\Big|\sum_{n=0}^{\infty}a_nX_nz^n\Big| \leq C\sqrt{\log\Big(\frac{1}{1-r}\log\frac{\mu_f(r)}{1-r}\Big)}S_f(r)
\]
for $r_0\leq r<1$, $r\notin E$.
\end{theorem}

\begin{proof}
After multiplying $f$ by a constant, if necessary, we may assume that $\lim_{r\to 1}\mu_f(r)>e$. 

Now let $E\subset\intervalco{0}{1}$ be the open set of finite logarithmic measure that is the union of the open set in
 Lemma \ref{rogHDELemmaprobaTail}, taken for some $\alpha>1$ and $\delta>0$, and the open set in Lemma \ref{l-skkurg} for the same $\delta$. Note that the complement of $E$ has 1 as limit point.

We apply Lemma \ref{rogHDELemmasubseq} to $R=1$, $\varphi(r)=\log \frac{1}{1-r}$, $\psi_1=\log\mu_f$ and $\psi_2=\log S_f$ with $0<\rho<1$ so large that $\varphi(\rho)\geq 1$, $\psi_1(\rho)\geq 1$ and $\psi_2(\rho)\geq 1$. Let $(r_{l,k,j})_{(l,k,j)\in J}$ be the family given by the lemma.

By \eqref{eq-wm2dg} we have that, for any $r\geq \rho$ and $r\notin E$,
\[
S_f(r) \leq G_f(r) \lesssim\Big(\frac{\mu_f(r)}{1-r}\Big)^{1+\delta}.
\]
Then we have by (i) and (ii) of Lemma \ref{rogHDELemmasubseq} that, for any $(l,k,j)\in J$, $e^j \lesssim e^{(1+\delta)(k+1)}e^{(1+\delta)(l+1)}$ and hence
\begin{equation}\label{eq-musg}
j \lesssim l+k.
\end{equation}

Define for each $(l,k,j)\in J$ the real number
\[
N_{l,k,j}:=\frac{1}{(1-r_{l,k,j})^2}\Big(\log\frac{\mu_f(r_{l,k,j})}{1-r_{l,k,j}}\Big)^{2+\delta}\geq 1
\]
and the set
\[
A_{l,k,j}:=\Big\{\Big\|\sum_{n=0}^{\infty}a_nX_nz^n\Big\|_{r_{l,k,j}}\geq C\sqrt{\log N_{l,k,j}}S_f(r_{l,k,j})\Big\},
\]
where $C>0$ is the constant of Lemma \ref{rogHDELemmaprobaTail}. Then (i) of Lemma \ref{rogHDELemmasubseq}, Lemma \ref{rogHDELemmaprobaTail}, the definition of $N_{l,k,j}$, (ii) of Lemma \ref{rogHDELemmasubseq} and \eqref{eq-musg} imply that
\begin{align*}
\sum_{(l,k,l)\in J}\mathbb{P}(A_{l,k,j}) \lesssim\sum_{(l,k,j)\in J}\frac{1}{N_{l,k,j}^{2\alpha}}
&=\sum_{(l,k,j)\in J}\frac{(1-r_{l,k,j})^{4\alpha}}{\big(\log\frac{\mu_f(r_{l,k,j})}{1-r_{l,k,j}}\big)^{2\alpha(2+\delta)}}\\
&\leq\sum_{(l,k,j)\in J}\frac{1}{e^{l4\alpha}(l+k)^{2\alpha(2+\delta)}}\\
&\lesssim \sum_{l,k\geq 1}\frac{l+k}{e^{l4\alpha}(l+k)^{2\alpha(2+\delta)}}<\infty.
\end{align*}

By the Borel-Cantelli lemma, we have that, for almost every $\omega\in\Omega$, there exist $l_0(\omega)$, $k_0(\omega)$, $j_0(\omega)\geq 1$ such that, for every $(l,k,j)\in J$, whenever $l > l_0(\omega)$ or $k > k_0(\omega)$ or $j> j_0(\omega)$ then
\begin{equation}\label{rogHDEROGsubsequence}
\Big\|\sum_{n=0}^{\infty}a_nX_n(\omega)z^n\Big\|_{r_{l,k,j}} \leq C\sqrt{\log N_{l,k,j}}S_f(r_{l,k,j}).
\end{equation}

We set $r_0(\omega):=\max_{l\leq l_0(\omega),k\leq k_0(\omega),j\leq j_0(\omega)}r_{l,k,j}$. Let $r>r_0(\omega)$ with $r\notin E$. By \emph{\ref{rogHDEsubsequenceLemma3}} of Lemma \ref{rogHDELemmasubseq}, there exists $(l,k,j)\in J$ such that $r\leq r_{l,k,j}$, $\varphi(r_{l,k,j})\leq \varphi(r)+1$, $\psi_1(r_{l,k,j})\leq \psi_1(r)+1$ and $\psi_2(r_{l,k,j})\leq \psi_2(r)+1$. Since $r>r_0(\omega)$ we must have either $l>l_0(\omega)$ or $k> k_0(\omega)$ or $j> j_0(\omega)$, hence \eqref{rogHDEROGsubsequence} holds. The Maximum Principle then implies that
\begin{align*}
\Big\|\sum_{n=0}^{\infty}a_nX_n(\omega)z^n\Big\|_{r} &\leq\Big\|\sum_{n=0}^{\infty}a_nX_n(\omega)z^n\Big\|_{r_{l,k,j}}\lesssim\sqrt{\log N_{l,k,j}}S_f(r_{l,k,j})\\
&\lesssim\sqrt{\log\Big(\frac{1}{1-r_{l,k,j}}\log\frac{\mu_f(r_{l,k,j})}{1-r_{l,k,j}}\Big)}S_f(r_{l,k,j})\\
&\lesssim\sqrt{\log\Big(\frac{1}{1-r}\log\frac{\mu_f(r)}{1-r}\Big)}S_f(r),
\end{align*}
which completes the proof.
\end{proof}

We can now prove Theorem \ref{t-main2} exactly as we proved Theorem \ref{t-main1} at the end of the previous section. For this, we estimate 
$\log\big(\frac{1}{1-r}\log\frac{\mu_f(r)}{1-r}\big)$ for large $r$ by $\big(\log\frac{1}{1-r}\big)\big(\log\frac{\mu_f(r)}{1-r}\big)^{\delta}$.

\begin{remark}\label{r-KS23}
Kuryliak and Skaskiv \cite[Theorem 1, Corollary 1]{KuryliakSkaskiv2023} obtain a weaker version of Theorem \ref{t-main2} under the additional assumption that, for some $\beta>0$ and some $N\geq 0$, $\sup_{n\geq N} E\big(|X_n|^{-\beta}\big)<\infty$, see \cite[(7)]{KuryliakSkaskiv2023}; the authors take the infimum instead of the supremum, but the proof of \cite[Proposition 1]{KuryliakSkaskiv2023} shows that this is a misprint. This additional assumption is not satisfied, for example, for the centred subgaussian sequence $(X_n)_n$ where $X_n$ is uniformly distributed on $[-\frac{1}{n+1},\frac{1}{n+1}]$, $n\geq 0$.
\end{remark}

\section{A unified result}\label{s-unif}
In this brief section we unify the results in the previous two sections and generalize them to other notions of exceptional sets. The results concern any analytic function in a disk $|z|<R$, $0<R\leq \infty$. The growth-related functions $M_f$, $\mu_f$, $S_f$, $G_f$ are defined as before.

\begin{definition} Let $0<R\leq \infty$. Let $h:\intervalco{\rho}{R}\to\intervalco{0}{\infty}$ be a continuous increasing function with $\int_{\rho}^R\frac{h(r)}{r}\emph{\intd} r=\infty$ for some $\rho\in \intervalco{0}{R}$. Then a set $E\subset \intervalco{0}{R}$ is said to be of \emph{finite $h$-logarithmic measure} if 
\[
\int_{E\cap \intervalco{\rho}{R}} \frac{h(r)}{r}\intd r<\infty.
\]
\end{definition}

See \cite{SkaskivKuryliak2020} and \cite{KuryliakSkaskiv2023} for this notion; the name seems to derive from the fact that $\frac{h(r)}{r}\intd r= h(r)\intd \ln r$. This generalizes the notion of logarithmic measure for $R=\infty$ (where $h$ is constant) and for $R=1$ (where $h(r)=\frac{1}{1-r}$).  

The following Wiman-Valiron inequality for an arbitrary $R\in\intervaloc{0}{\infty}$ is essentially due to Skaskiv and Kuryliak \cite{SkaskivKuryliak2020}: see the penultimate inequality in the proof of their Theorem 1 and note that, for any $\varepsilon>0$, $\max(a,b)\lesssim ab$ if $a,b\geq \varepsilon$. Another proof can be given with \cite[Theorem 2.1]{GrosseErdmann2024}: see Remark 2.4(c) and the discussion after Theorem 1.6 there.

\begin{theorem}[Skaskiv, Kuryliak]\label{t-SkaskivKuryliak}
Let $0<R\leq \infty$. Let $f(z)=\sum_{n=0}^\infty a_nz^n$ be a non-constant analytic function for $|z|<R$. Let $h:\intervalco{\rho}{R}\to\intervalco{0}{\infty}$ be a continuous increasing function with $\int_{\rho}^R\frac{h(r)}{r}\emph{\intd} r=\infty$ for some $\rho\in \intervalco{0}{R}$; suppose that $\lim_{r\to R} h(r) >1$ and $\lim_{r\to R} h(r)\mu_f(r) >e$.

Then, for every $\delta>0$, there exists a constant $C>0$ and an open set $E\subset\intervalco{0}{R}$ of finite $h$-logarithmic measure such that
\[
M_f(r)\leq C\, h(r) \mu_f(r)\big(\log h(r)\big)^{\frac{1}{2}+\delta}\big(\log (h(r)\mu_f(r))\big)^{\frac{1}{2}}\big(\log\log (h(r)\mu_f(r))\big)^{1+\delta}
\]
for every $r\in (\rho,R)$, $r\notin E$.
\end{theorem}

The additional assumption on $h$ is only needed in order to make sure that the inequality has a sense for large $r$.

Then, based on the results in Section \ref{s-prelim} and proceeding exactly as in the proofs in Section \ref{s-LevyD}, we obtain the following, which contains Theorems \ref{rogHCEThmMain} and \ref{rogHDEThmMain} as special cases.

\begin{theorem}\label{t-mufsf}
Let $0<R\leq \infty$. Let $f(z)=\sum_{n=0}^\infty a_nz^n$ be a non-constant analytic function for $|z|<R$. Let $h:\intervalco{\rho}{R}\to\intervalco{0}{\infty}$ be a continuous increasing function with $\int_{\rho}^R\frac{h(r)}{r}\emph{\intd} r=\infty$ for some $\rho\in \intervalco{0}{R}$; suppose that $\lim_{r\to R} h(r) >1$ and $\lim_{r\to R} h(r)\mu_f(r) >e$. Let $(X_n)_{n\geq 0}$ be an independent centred subgaussian sequence. 

Then $\sum_{n=0}^\infty a_nX_nz^n$ defines almost surely an analytic function for $|z|<R$. Moreover, there exists an open set $E\subset\intervalco{0}{R}$ of finite $h$-logarithmic measure and a constant $C>0$ such that, almost surely, there exists some $r_0\in\intervaloo{\rho}{R}$ such that
\[
\max_{|z|=r}\Big|\sum_{n=0}^{\infty}a_nX_nz^n\Big| \leq C\sqrt{\log\big(h(r)\log (h(r)\mu_f(r))\big)}S_f(r)
\]
for $r_0\leq r<R$, $r\notin E$.
\end{theorem}

With the usual procedure we then obtain the following, which contains Theorems \ref{t-main1} and \ref{t-main2} as special cases.

\begin{theorem}\label{t-mufmuf}
Let $0<R\leq \infty$. Let $f(z)=\sum_{n=0}^\infty a_nz^n$ be a non-constant analytic function for $|z|<R$. Let $h:\intervalco{\rho}{R}\to\intervalco{0}{\infty}$ be a continuous increasing function with $\int_{\rho}^R\frac{h(r)}{r}\emph{\intd} r=\infty$ for some $\rho\in \intervalco{0}{R}$; suppose that $\lim_{r\to R} h(r) >1$ and $\lim_{r\to R} h(r)\mu_f(r) >1$. Let $(X_n)_{n\geq 0}$ be an independent centred subgaussian sequence. 

Then $\sum_{n=0}^\infty a_nX_nz^n$ defines almost surely an analytic function for $|z|<R$. Moreover, for every $\delta > 0$, there exists an open set $E\subset\intervalco{0}{R}$ of finite $h$-logarithmic measure and a constant $C>0$ such that, almost surely, there exists some $r_0\in\intervaloo{\rho}{R}$ such that
\[
\max_{|z|=r}\Big|\sum_{n=0}^{\infty}a_nX_nz^n\Big| \leq C\, h(r)^{\frac{1}{2}}\mu_f(r)\big(\log h(r)\big)^{\frac{3}{4}+\delta} 
\big(\log (h(r)\mu_f(r))\big)^{\frac{1}{4}}\big(\log\log (h(r)\mu_f(r))\big)^{1+\delta}
\]
for $r_0\leq r<R$, $r\notin E$.
\end{theorem}

\section{An application to linear dynamics}\label{s-lindyn}

Our results have an immediate application in linear dynamics, the study of dynamical properties of linear operators. We briefly introduce the necessary background. We consider the vector spaces $\mathcal{X}=H(\mathbb{C})$ of entire functions and $\mathcal{X}=H(\mathbb{D})$ of analytic functions on the unit disk $\mathbb{D}$, both endowed with the topology of uniform convergence on compact sets. A (continuous and linear) operator $T:\mathcal{X}\to \mathcal{X}$ is called \textit{hypercyclic} if it has a dense orbit, that is, if there is some function $f\in \mathcal{X}$ whose orbit $\{T^nf : n\geq 0\}$ is dense in $\mathcal{X}$; such a function $f$ is then called \textit{hypercyclic} for $T$. More restrictively, the function $f$ is called \textit{frequently hypercyclic} for $T$ if, for every non-empty open set $U\subset \mathcal{X}$, the set of return times to $U$ has positive lower density, that is,
\[
\underline{\text{dens}}\{n\geq 0: T^nf\in U\}>0,
\]
where $\underline{\text{dens}}\, A = \liminf_{N\to\infty} \frac{1}{N+1}\text{card}\{n\in A: 0\leq n\leq N\}$ for $A\subset \mathbb{N}_0$. An operator that admits a frequently hypercyclic function is itself called \textit{frequently hypercyclic}. A related notion is that of \textit{chaos}, where $T$ is supposed to be hypercyclic and possess a dense set of periodic points, that is, functions $f\in \mathcal{X}$ such that $T^nf=f$ for some $n\geq 1$. For introductions to linear dynamics, see \cite{BayartMatheron2009} and \cite{GrosseErdmannPeris2011}.

Now, a \textit{weighted shift operator} $B_w$ on $\mathcal{X}$ is an operator that maps the analytic function $f(z)=\sum_{n=0}^\infty a_nz^n$ to
\[
(B_wf)(z) = \sum_{n=0}^\infty w_{n+1} a_{n+1}z^n,
\]
where $w=(w_n)_{n\geq 1}$ is a \textit{weight}, that is a sequence of non-zero complex numbers. It is well known that $B_w$ is chaotic on $\mathcal{X}=H(\mathbb{C})$ or $H(\mathbb{D})$ if and only if
\[
\sum_{n=0}^\infty \frac{1}{\prod_{k=1}^n w_k} z^n \in \mathcal{X};
\]
and in that case $B_w$ is frequently hypercyclic, see \cite[Section 4.1 and Corollary 9.14]{GrosseErdmannPeris2011}. 

The best known examples are the differentiation operator $D$ on $H(\mathbb{C})$, where $w=(n)_n$, and the so-called Taylor shift $T$ on $H(\mathbb{D})$, where $w=(1)_n$; in other words,
\[
Df(z)=f'(z) \quad \text{and}\quad Tf(z) = \tfrac{f(z)-f(0)}{z}\ (z\neq 0), Tf(0)=f'(0).
\]
By the above criterion, both operators are frequently hypercyclic.

An interesting problem in this context is that of finding the least possible rate of growth of functions $f$ that are hypercyclic or frequently hypercyclic for a given weighted shift $B_w$. This line of research was initiated by the second author \cite{GrosseErdmann1990} and by Shkarin \cite{Shkarin1993} in the case of hypercyclic functions for the differentiation operator $D$; they showed that, for any function $\phi : (0,\infty) \to [1,\infty)$ with $\lim_{r\to\infty}\phi(r) = \infty$ there exists an entire function $f$ that is hypercyclic for $D$ and a constant $C>0$ such that
\[
M_f(r) \leq C\phi(r)\frac{e^r}{r^{\frac{1}{2}}}
\]
for every $r>0$. And this is optimal in the sense that $\phi$ cannot be dropped. 

The corresponding result for frequent hypercyclicity is due to Drasin and Saksman \cite{DrasinSaksman2012}; they have shown that there exists an entire function $f$ that is frequently hypercyclic for $D$ and a constant $C>0$ such that
\[
M_f(r) \leq C\frac{e^r}{r^{\frac{1}{4}}}
\]
for every $r>0$; this is again optimal in a certain sense. 

Now, the latter result was considerably more demanding than that for hypercyclicity. This motivated Nikula \cite{Nikula2014} to use a probabilistic approach. He assumed $X$ to be a centred subgaussian complex random variable of full support, that is, for any non-empty open set $U\subset \mathbb{C}$ we have that $\mathbb{P}(X\in U)>0$. If $(X_n)_{n\geq 0}$ is a sequence of i.i.d. copies of $X$, then
\[
g(z):=\sum_{n=0}^\infty \frac{X_n}{n!} z^n
\]
defines almost surely an entire function that is frequently hypercyclic for $D$ and for which there exists $r_0>0$ such that
\[
\|g\|_r \leq C\sqrt{\log r}\frac{e^r}{r^{\frac{1}{4}}}
\]
for every $r\geq r_0$. In other words, his probabilistic method led to an extra factor of $\sqrt{\log r}$; see also \cite[Proposition 6]{Nikula2014}.

Similar results for the Taylor shift on $H(\mathbb{D})$ are due to Mouze and Munnier \cite{MouzeMunnier2024} (hypercyclic case), \cite{MouzeMunnier2021b} (frequently hypercyclic case), and \cite[p.\ 627]{MouzeMunnier2021} (frequently hypercyclic case with a probabilistic approach), where the respective rates of growth are of the form
\[
C\phi(r),\quad C\frac{1}{\sqrt{1-r}}\quad\text{and}\quad C\sqrt{\log\frac{1}{1-r}}\frac{1}{\sqrt{1-r}}.
\] 
For further results on other weighted shift operators see \cite{Agneessens2024} and the literature cited there.  

In all of these results, the rate of growth holds for all sufficiently large $r$. The question of a rate of growth with an exceptional set has not been addressed yet in linear dynamics. We can deduce such results immediately from the work in this paper and the following result of the first author \cite[Theorem 4.4]{Agneessens2023a}.

\begin{theorem}
Let $T=B_w$ be a chaotic weighted shift operator on $\mathcal{X}=H(\mathbb{C})$ or $H(\mathbb{D})$ with weight $w=(w_n)_{n\geq 1}$. Let $X$ be a subgaussian complex random variable of full support, and let $(X_n)_{n\geq 0}$ be a sequence of i.i.d. copies of $X$. Then 
\[
g(z):=\sum_{n=0}^\infty \frac{X_n}{\prod_{k=1}^n w_k} z^n
\]
defines almost surely a function from $\mathcal{X}$ that is frequently hypercyclic for $B_w$.
\end{theorem}

Incidentally, the assumption of a full support is crucial. Otherwise there would be a non-empty open set $U\subset\mathcal{X}$ so that 
\[
\mathbb{P}(\exists n\geq 0 : (B_w^ng)(0)\in U) = \mathbb{P}(\exists n\geq 0 : X_n\in U)=0.
\]
Hence, $g$ would almost surely not even be hypercyclic for $B_w$. 

Combining the result above with Theorem \ref{rogHCEThmMain}, we obtain the following.

\begin{theorem}\label{t-fhcrate1}
 Let $T=B_w$ be a chaotic weighted shift operator on $H(\mathbb{C})$ with weight $w=(w_n)_{n\geq 1}$. Let $X$ be a centred subgaussian complex random variable of full support, and let $(X_n)_{n\geq 0}$ be a sequence of i.i.d. copies of $X$. Then 
\[
g(z):=\sum_{n=0}^\infty \frac{X_n}{\prod_{k=1}^n w_k} z^n
\]
defines almost surely an entire function that is frequently hypercyclic for $B_w$. Moreover, there exists a set $E\subset\intervalco{0}{\infty}$ of finite logarithmic measure and a constant $C>0$ such that, almost surely, there exists $r_0>0$ such that
\[
\|g\|_r \leq C\sqrt{\log\log \mu_{f}(r)}S_f(r)
\]
for every $r\geq r_0$, $r\notin E$; here, $f$ is the entire function given by $f(z)=\sum_{n=0}^\infty \frac{1}{\prod_{k=1}^n w_k} z^n$.
\end{theorem}

In the same way, we obtain the following with Theorem \ref{rogHDEThmMain}.

\begin{theorem}\label{t-fhcrate2}
 Let $T=B_w$ be a chaotic weighted shift operator on $H(\mathbb{D})$ with weight $w=(w_n)_{n\geq 1}$. Let $X$ be a centred subgaussian complex random variable of full support, and let $(X_n)_{n\geq 0}$ be a sequence of i.i.d. copies of $X$. Then 
\[
g(z):=\sum_{n=0}^\infty \frac{X_n}{\prod_{k=1}^n w_k} z^n
\]
defines almost surely an analytic function on $\mathbb{D}$ that is frequently hypercyclic for $B_w$. Moreover, there exists a set $E\subset\intervalco{0}{1}$ of finite logarithmic measure and a constant $C>0$ such that, almost surely, there exists $r_0\in (0,1)$ such that
\[
\|g\|_r \leq C \sqrt{\log\Big(\frac{1}{1-r}\log\frac{\mu_f(r)}{1-r}\Big)}S_f(r)
\]
for $r_0\leq r<1$, $r\notin E$; here, $f\in H(\mathbb{D})$ is given by $f(z)=\sum_{n=0}^\infty \frac{1}{\prod_{k=1}^n w_k} z^n$.
\end{theorem}

By way of an example, let us look again at the two operators mentioned above. 

For the differentiation operator $D$ on $H(\mathbb{C})$ we have that $f(z)=e^z$, and it is well known that $\mu_f(r)\asymp r^{-\frac{1}{2}}e^r$ and $S_f(r)\asymp r^{-\frac{1}{4}}e^r$, see also \cite{Agneessens2024}. Thus
\[
\|g\|_r \leq C \sqrt{\log r}\frac{e^r}{r^\frac{1}{4}}
\]
almost surely, outside a set of finite logarithmic measure.

For the Taylor shift $T$ on $H(\mathbb{D})$ we have that $f(z)=\frac{1}{1-z}$, so that $\mu_f(r)= 1$ and $S_f(r)\asymp \frac{1}{\sqrt{1-r}}$. Thus
\[
\|g\|_r \leq C \sqrt{\log\frac{1}{1-r}}\frac{1}{\sqrt{1-r}}
\]
almost surely, outside a set of finite logarithmic measure.

We see here that our results give less than those of Nikula, and of Mouze and Munnier, who obtain the same inequalities for all large values of $r$. But Theorems \ref{t-fhcrate1} and \ref{t-fhcrate2} above hold for all chaotic weighted shifts.

We refer to the forthcoming paper \cite{Agneessens2024} for rate of growth results without exceptional sets for large classes of chaotic weighted shifts.

\end{document}